\documentclass[11pt]{amsart}

\usepackage{amssymb,amsmath,epsfig,graphics}

\numberwithin{equation}{section}
\newcommand{\beq}{\begin{equation}}
\newcommand{\eeq}{\end{equation}}
\newcommand{\bea}{\begin{eqnarray}}
\newcommand{\eea}{\end{eqnarray}}
\newcommand{\beas}{\begin{eqnarray*}}
\newcommand{\eeas}{\end{eqnarray*}}

\newtheorem{theorem}{Theorem}[section]

\newtheorem{proposition}[theorem]{Proposition}

\newtheorem{remark}[theorem]{Remark}
\newtheorem{example}[theorem]{Example}
\newtheorem{examples}[theorem]{Examples}
\newtheorem{foo}[theorem]{Remarks}

\newcommand{\bM}{\mathbb M}

\newcommand{\ve}{\varepsilon}

\title{Sub-Riemannian balls in CR Sasakian manifolds}

\author{Fabrice Baudoin}
\address{Department of Mathematics\\Purdue University \\
West Lafayette, IN 47907} \email[Fabrice Baudoin]{fbaudoin@math.purdue.edu}
\thanks{First author supported in part by
NSF Grant DMS 0907326}

\author{Michel Bonnefont}

\begin{document}

\maketitle

\begin{abstract}
We prove global estimates for the sub-Riemannian distance of CR Sasakian manifolds with non negative horizontal Webster-Tanaka Ricci curvature. In particular, in this setting,  large sub-Riemannian balls are comparable to Riemannian balls.
\end{abstract}

\tableofcontents

\section{Introduction}

 Let $\mathbb{M}$ be a complete strictly pseudo convex  CR Sasakian manifold with real dimension $2n +1$. Let $\theta$ be a pseudo-hermitian form on $\mathbb{M}$ with respect to which the Levi form $L_\theta$ is positive definite. The kernel of $\theta$ determines an horizontal bundle $\mathcal H$. Denote now  by $T$ the Reeb vector field on $\bM$, i.e., the characteristic direction of $\theta$. We denote by $\nabla$ the Tanaka-Webster connection of $\mathbb{M}$.

We recall that the \emph{CR} manifold $(\bM,\theta)$ is called Sasakian if the pseudo-hermitian torsion of $\nabla$ vanishes, in the sense that $\mathbf{T}(T,X) = 0$, for every $X\in \mathcal H$.
For instance the standard CR structures on the  Heisenberg group $\mathbb{H}_{2n+1}$ and the sphere $\mathbb{S}^{2n+1}$ are Sasakian. In every Sasakian manifold the Reeb vector field $T$ is a sub-Riemannian Killing vector field (see Theorem 1.5 on p. 42 and Lemma 1.5 on p. 43  in \cite{CR}).   

 We consider the family of scaled Riemannian metrics $g_\tau$, $\tau >0$, such that for $X,Y \in \mathcal{H}$:
\begin{align}\label{scaled}
g_\tau (X,Y)=d\theta (X,JY), \quad g_\tau (X,T)=0, \quad g_\tau (T,T) =\frac{1}{\tau^2}.
\end{align}
where $J$ is the complex structure on $\mathbb{M}$. We denote by $d_\tau$ the distance corresponding to the Riemannian structure $g_\tau$ and by $d$ the sub-Riemannian distance on $\mathbb{M}$. It is well known that $d_\tau(x,y) \to d(x,y)$, when $\tau \to 0$.
Our goal is to prove the following theorem:

\begin{theorem}\label{Global}
Let $\mathcal{R}$ be the Ricci curvature of the Webster-Tanaka connection $\nabla$. If for every $X \in \mathcal{H}$,
\[
\mathcal{R} (X,X) \ge 0,
\]
then for every $x,y \in \mathbb{M}$,
 \[
d_\tau(x,y) \le d(x,y) \le A_n d_\tau(x,y) +B_n \sqrt{\tau} d_\tau(x,y)^{1/2},
\]
where $A_n$ and $B_n$ are two positive universal constants depending only on $n$.
\end{theorem}

To put things in perspective, estimates between the sub-Riemannian distance and Riemannian ones have been extensively studied in the litterature (see for instance \cite{FP},  \cite{Gromov2}, \cite{JSC}, \cite{NSW}, \cite{RS}). But in these cited works, such estimates are local in nature. To the knowledge of the authors Theorem \ref{Global} is  the first result that gives global and uniform estimates for a large class of sub-Riemannian metrics. It is consistent with the well known  Nagel-Stein-Wainger estimate \cite{NSW}, that implies at small scales $d(x,y) \le  C d_\tau (x,y)^{\frac{1}{2}}$ and shows that due to curvature effects at big scales we have  $d(x,y) \simeq d_\tau (x,y)$.

\section{Li-Yau and Harnack estimates for the heat kernel on Sasakian manifolds}

 \subsection{Curvature dimension inequalities  and heat kernel bounds}

We first recall some results that will be needed in the sequel and that can be found in \cite{BBG} and \cite{BG}.
We denote by $\Delta$ the sub-Laplacian on $\mathbb{M}$ and by $\nabla^\mathcal{H}$ the horizontal gradient. For smooth functions $f:\mathbb{M} \to \mathbb{R}$, set
\begin{equation}\label{gamma2}
\Gamma_{2}(f) = \frac{1}{2}\big[\Delta \| \nabla^\mathcal{H} f \|^2 - 2 \langle \nabla^\mathcal{H} f,
\nabla^\mathcal{H} \Delta f \rangle\big],
\end{equation}
and
\begin{equation}\label{gamma2Z}
\Gamma^T_{2}(f) = \frac{1}{2}\big[\Delta (Tf)^2 -2(Tf)(T\Delta f)\big].
\end{equation} 

The following result was obtained in \cite{BG} by means of a Bochner's type formula.

\begin{theorem}\label{T:sasakian}
Assume that  for every $X \in \mathcal{H}$,
\[
\mathcal{R} (X,X) \ge 0.
\]
Then for every $f\in C^\infty(\bM)$ and any $\nu>0$, 
 \[
 \Gamma_{2}(f)+\nu \Gamma^T_{2}(f) \ge \frac{1}{2n} (\Delta f)^2  -\frac{1}{\nu} \| \nabla^\mathcal{H} f \|^2
 +\frac{n}{2} (Tf)^2.
 \]
\end{theorem}

We denote by $p(t,x,y)$ the heat kernel of $\mathbb{M}$, that is the fundamental solution of the heat equation $\frac{\partial f}{\partial t}=\Delta f$. The following global lower and upper bounds were proved in \cite{BBG}.

\begin{theorem}\label{gb}
Assume that  for every $X \in \mathcal{H}$,
\[
\mathcal{R} (X,X) \ge 0.
\]
For any $0<\ve \le 1$
there exists a constant $C(\ve) = C(n,\ve)>0$, which tends
to $\infty$ as $\ve \to 0^+$, such that for every $x,y\in \bM$
and $t>0$ one has
\[
\frac{C(\ve)^{-1}}{\mu(B(x,\sqrt
t))} \exp
\left(-\left( 1+\frac{3}{n} \right)\frac{ d(x,y)^2}{(4-\ve)t}\right)\le p(t,x,y)\le \frac{C(\ve)}{\mu(B(x,\sqrt
t))} \exp
\left(-\frac{d(x,y)^2}{(4+\ve)t}\right).
\]
\end{theorem}

In the above Theorem, $d$ is the sub-Riemannian distance, $B(x,\sqrt
t)$ is the sub-Riemannian ball with center $x$ and radius $\sqrt{t}$ and $\mu$ is the volume corresponding to the volume form $\theta \wedge (d\theta)^n$.

\subsection{Harnack type estimates}

From now on and in all the sequel we assume that  for every $X \in \mathcal{H}$, $\mathcal{R} (X,X) \ge 0$.
We first have the following Li-Yau type estimate for the heat kernel.
\begin{proposition}\label{Li-Yau}
For $t>0$,
\begin{align*}
\| \nabla^\mathcal{H} \ln p_t \|^2 +\frac{n}{3}  t  (T \ln p_t )^2  \le
\left(1+\frac{3}{n}\right)
\frac{\Delta p_t }{p_t}  +\frac{n \left( 1+\frac{3}{n}
\right)^2}{t}.
\end{align*}
\end{proposition}

\begin{proof}
The result is essentially proved in \cite{BG}, but due to the simplicity of the argument we reproduce, without the details, the proof by sake of completeness. Fix $T>0$ and consider the functional
\[
\Phi(t)=\frac{3}{n} (T-t)^2 P_t \left( \frac{\| \nabla^\mathcal{H} p_{T-t} \|^2}{p_{T-t}}\right) +(T-t)^3  P_t \left( \frac{ (Tp_{T-t})^2}{p_{T-t}}\right),
\]
where $P_t$ is the heat semigroup associated with $\Delta$. Since $T$ is a Killing vector field, for any smooth function $f$ we have
\[
\langle \nabla^\mathcal{H} f,  \nabla^\mathcal{H} (Tf)^2  \rangle= (Tf) (T  \|  \nabla^\mathcal{H} f \|^2).
\]
Differentiating $\Phi$  and using the above yields 
\begin{align*}
\Phi'(t)&=  \frac{6}{n} (T-t)^2 P_t \left(p_{T-t}\Gamma_2( \ln p_{T-t}) \right) +2(T-t)^3  P_t \left(p_{T-t} \Gamma_2^T(\ln p_{T-t})  \right) \\ 
 &-\frac{6}{n} (T-t) P_t \left( \frac{\| \nabla^\mathcal{H} p_{T-t} \|^2}{p_{T-t}}\right) -3(T-t)^2  P_t \left( \frac{ (Tp_{T-t})^2}{p_{T-t}}\right).
\end{align*}
From Theorem \ref{T:sasakian}, we have
\begin{align*}
 & \frac{6}{n} (T-t)^2 p_{T-t}\Gamma_2( \ln p_{T-t}) +2(T-t)^3 p_{T-t} \Gamma_2^T(\ln p_{T-t}) \\
 \ge & \frac{3}{n^2} (T-t)^2 p_{T-t} (\Delta \ln p_{T-t})^2-\frac{18}{n^2} (T-t) p_{T-t} \|  \nabla^\mathcal{H} \ln p_{T-t} \|^2 +3(T-t)^2 p_{T-t} (T \ln p_{T-t})^2.
\end{align*}
Therefore we obtain
\begin{align*}
\Phi'(t)& \ge \frac{3}{n^2} (T-t)^2P_t( p_{T-t} (\Delta \ln p_{T-t})^2)-\left(\frac{18}{n^2} +\frac{6}{n}\right) (T-t) P_t(p_{T-t} \|  \nabla^\mathcal{H} \ln p_{T-t} \|^2 ).
\end{align*}
Now, for every $\gamma(t)$, we have
\begin{align*}
(\Delta (\ln p_{T-t}) )^2 & \ge 2\gamma(t) \Delta \ln p_{T-t} -\gamma(t)^2  \\
 & \ge 2\gamma(t) \left( \frac{\Delta p_{T-t}}{p_{T-t}} - \| \nabla^\mathcal{H} \ln p_{T-t} \|^2\right) -\gamma(t)^2.
\end{align*}
Therefore we get
\begin{align*}
\frac{3}{n^2} (T-t)^2P_t( p_{T-t} (\Delta \ln p_{T-t})^2)  \ge  \frac{6}{n^2} (T-t)^2\gamma(t) \left( \Delta p_{T} -P_t\left(  p_{T-t} \| \nabla^\mathcal{H} \ln p_{T-t} \|^2\right) \right) -\frac{3}{n^2} (T-t)^2\gamma(t)^2p_T.
\end{align*}
This implies
\[
\Phi'(t) \ge  \frac{6}{n^2} (T-t)^2\gamma(t)  \Delta p_{T} -\frac{3}{n^2} (T-t)^2\gamma(t)^2p_T-\left(\frac{18}{n^2} +\frac{6}{n}+\frac{6}{n^2}(T-t)\gamma(t) \right) (T-t) P_t(p_{T-t} \|  \nabla^\mathcal{H} \ln p_{T-t} \|^2 ).
\]
Setting $\gamma(t)=-\frac{n+3}{T-t}$, leads then to
\[
\Phi'(t) \ge  - \frac{6(n+3)}{n^2}(T-t) \Delta p_{T}-\frac{3}{n^2}(n+3)^2.
\]
By integrating the last inequality from 0 to $T$, we obtain
\[
-\Phi(0) \ge  - \frac{3(n+3)}{n^2}T^2 \Delta p_{T}-\frac{3}{n^2}(n+3)^2 T,
\]
which is the required inequality.
\end{proof}

We can deduce from the previous Li-Yau type inequality the following Harnack inequality.

\begin{theorem} \label{harnack} For $x,y,z \in \mathbb{M}$, $s<t$,
\begin{equation*}
p(s,x,y) \le  p(t,x,z) \left(\frac{t}{s}\right)^{n+3}  \exp\left( \left( 1+\frac{3}{n}\right) \left( \frac{1}{4(t-s)} +\frac{  \frac{3}{n}\tau^2 \ln\frac{t}{s}} {4(t-s)^2} \right) d_\tau(x,y)^2\right), \quad s<t,
\end{equation*}
where $d_\tau$ denotes the Riemannian metric introduced in (\ref{scaled}).
\end{theorem}

\begin{proof}

From Proposition \ref{Li-Yau},
\[
\| \nabla^\mathcal{H} \ln p_t \|^2  \le
\left(1+\frac{3}{n}\right)
\frac{\Delta p_t }{p_t}  +\frac{n \left( 1+\frac{3}{n}
\right)^2}{t}.
\]
and 
\[
\frac{n}{3}  t  (T \ln p_t )^2  \le
\left(1+\frac{3}{n}\right)
\frac{\Delta p_t }{p_t}  +\frac{n \left( 1+\frac{3}{n}
\right)^2}{t}.
\]

Therefore we have that for every $\tau >0$,
\begin{align}\label{hj}
\| \nabla^\mathcal{H} \ln p_t \|^2+ \tau^2 (T \ln p_t )^2 \le \left(1+\frac{3\tau^2}{n t}\right) \left(1+\frac{3}{n}\right) \frac{\Delta p_t}{p_t} + \frac{n \left( 1+\frac{3}{n}
\right)^2}{t}\left( 1+ \frac{3\tau^2}{n t}\right),
\end{align}

Let now  $x,y,z \in \bM$ and let $\gamma:[s,t] \to \bM$, $s<t$, be an absolutely continuous path such that $\gamma(s)=y, \gamma(t)=z$.
We first write (\ref{hj}) in the form
\begin{align}\label{LIYAU2}
g_\tau( \nabla^\tau \ln p_u, \nabla^\tau \ln p_u ) \le a(u) \frac{  \Delta p_u}{p_u} +b(u),
 \end{align}
 where 
 \[
 g_\tau( \nabla^\tau \ln p_u, \nabla^\tau \ln p_u ) =\| \nabla^\mathcal{H} \ln p_t \|^2+ \tau^2 (T \ln p_t )^2 
 \]
 and
 \[
 a(u)=\left(1+\frac{3\tau^2}{n u}\right) \left(1+\frac{3}{n}\right), 
 \]
 \[
 b(u)=  \frac{n \left( 1+\frac{3}{n}
\right)^2}{u}\left( 1+ \frac{3\tau^2}{n u}\right).
 \]
 Let us now consider
 \[
 \phi(u)=\ln p_u(x,\gamma(t)).
 \]
 We  compute
 \[
  \phi'(u)= ( \partial_u \ln p_u  (x,\gamma(u))+g_\tau( \nabla^\tau \ln p_u  (x, \gamma(u)),\gamma'(u) ).
 \]
 Now, for every $\lambda >0$, we have
 \[
g_\tau( \nabla^\tau \ln p_u  (x, \gamma(u)),\gamma'(u) )\ge -\frac{1}{2\lambda^2}  g_\tau( \nabla^\tau \ln p_u, \nabla^\tau \ln p_u ) -\frac{\lambda^2}{2} g_\tau(\gamma'(u),\gamma'(u)).
 \]
 Choosing  $\lambda=\sqrt{\frac{a(u)}{2} }$ and using then (\ref{LIYAU2}) yields
 \[
 \phi'(u) \ge -\frac{b(u)}{a(u)} -\frac{1}{4} a(u)g_\tau(\gamma'(u),\gamma'(u)).
 \]
 By integrating this inequality from $s$ to $t$ we get as a result.
 \[
 \ln p(t,x,y)-\ln p(s,x,z)\ge -\int_s^t \frac{b(u)}{a(u)} du  -\frac{1}{4} \int_s^t a(u) g_\tau(\gamma'(u),\gamma'(u)) du.
 \]
We now minimize the quantity  $\int_s^t a(u) \| \gamma'(u) \|^2 du$ over the set of absolutely continuous paths such that $\gamma(s)=y, \gamma(t)=z$. By using reparametrization of paths, it is seen that
 \[
 \int_s^t a(u) \| \gamma'(u) \|^2 du \ge \frac{d^2(x,y)}{\int_s^t \frac{dv}{a(v)}},
 \]
 with equality achieved for $\gamma(u)=\sigma\left( \frac{\int_s^u \frac{dv}{a(v)}}{\int_s^t \frac{dv}{a(v)}} \right)$ where $\sigma:[0,1] \to \bM$ is a unit geodesic joining $y$ and $z$. As a conclusion we get
 \[
p(s,x,y) \le \exp\left( \int_s^t \frac{b(u)}{a(u)} du + \frac{d_\tau^2(y,z)}{4\int_s^t \frac{dv}{a(v)}}  \right) p(t,x,z).
 \]
 Finally, from Cauchy-Schwarz inequality, we have
 \[
 \int_s^t \frac{dv}{a(v)} \ge \frac{(t-s)^2}{\int_s^t a(v)dv}
 \]
 and thus
 \[
p(s,x,y) \le \exp\left( \int_s^t \frac{b(u)}{a(u)} du + \frac{d_\tau^2(y,z)\int_s^t a(v)dv}{4(t-s)^2}  \right) p(t,x,z).
 \]
 
\end{proof}

\section{Uniform distance estimates}

We are now ready to prove Theorem \ref{Global}:

\begin{proof}
The inequality $d_\tau(x,y) \le d(x,y)$ is straightforward. We prove now the second inequality. From Theorem \ref{harnack} and Theorem \ref{gb},
 \begin{align*}
p(t,x,y)  & \ge \frac{1}{2^{n+3}} p(t/2,x,x) \exp\left(- \left( 1+\frac{3}{n}\right) \left( \frac{1}{2t} +\frac{  3 \ln 2\tau^2 } {nt^2} \right) d_\tau(x,y)^2\right) \\
           &\ge\frac{1}{2^{n+3}} \frac{C_0(n) }{\mu(\mathbf{B}(x,\sqrt{t})) } \exp\left(- \left( 1+\frac{3}{n}\right) \left( \frac{1}{2t} +\frac{  3 \ln 2\tau^2 } {nt^2} \right) d_\tau(x,y)^2\right)\\
\end{align*}

From the Gaussian upper bound of  Theorem \ref{gb} and the previous lower bound, we deduce that for all $t>0$,
$$
\ln \frac{2^{n+3}C(\ve)}{C_0(n)} +  \left( \frac{1}{2}\left( 1+\frac{3}{n}\right)  d_\tau(x,y)^2 -\frac{d(x,y)^2}{4+\ve}\right)\frac{1}{t} +\left( 1+\frac{3}{n}\right) \left(\frac{3 \ln 2}{n}\right)\tau^2 d_\tau^2(x,y) \frac{1}{t^2}\geq 0
$$
We now chose $t=\tau d_\tau(x,y)$ and obtain
\[
d(x,y)^2 \le (4+\ve) \left( \ln \frac{2^{n+3}C(\ve)}{C_0(n)} + \left( 1+\frac{3}{n}\right) \left(\frac{3 \ln 2}{n}\right) \right) \tau d_\tau (x,y)+\left( 2+\frac{\ve}{2} \right) \left( 1+\frac{3}{n}\right) d_\tau(x,y)^2.
\]
\end{proof}

\end{document}